\documentclass[12pt,a4paper]{article}

\usepackage{amsthm}
\usepackage{epsfig}
\usepackage{alltt}
\usepackage{makeidx}
\usepackage{newlfont}
\usepackage{amsmath}
\usepackage{amssymb}
\usepackage{amsfonts}
\usepackage{amscd}
\usepackage[ngerman]{babel}
\usepackage{hyperref}
\usepackage{verbatim}
\usepackage{changebar}
\input{xy}
\xyoption{all}

\makeindex
\newtheorem*{Satz}{Satz}%[section]
\newtheorem{Lemma}{Lemma}

\theoremstyle{definition}
\newtheorem*{Definition}{Definition}

\theoremstyle{remark}
\newtheorem*{Bemerkung}{Bemerkung}

\newtheorem*{Ubung}{\"{U}bung}

\newcommand{\op}{\operatorname}

\newcommand{\ra}{\rightarrow}

\newcommand{\RA}{\Rightarrow}

\newcommand{\Bl}[1]{{\mathbb{#1}}}

\newcommand{\DR}{\Bl{R}}

\begin{document}
\title{Herleitung von Skalarprodukten aus 
Symmetrieprinzipien}
\author{Wolfgang Soergel}
%\address{Universit\"at Freiburg\\Mathematisches Institut
%\\Eckerstra\ss e 1\\ D-79104 Freiburg\\Germany}
%\email{soergel@@mathematik.uni-freiburg.de}
%\thanks{partially supported by the 
%TMR Algebraic Lie Theory ERB FMRX-CT97-0100}
\maketitle 
 Um den uns umgebenden Raum mathematisch zu modellieren, mag man
  von einem dreidimensionalen reellen affinen Raum 
ausgehen.  Um auch
  das Messen mit Zollst"ocken in dieses Modell zu integrieren, 
mag man annehmen, da"s zus"atzlich 
eine alle Translationen umfassende Untergruppe
  der Automorphismengruppe unseres affinen Raums 
  vorgegeben sei, deren Elemente  ``Bewegungen'' hei"sen m"ogen, 
und von der wir das folgende fordern: F"ur je zwei Paare von Teilmengen 
unseres affinen Raums
  bestehend aus einer Halbebene und einer Halbgerade auf ihrem Rand 
soll es genau
  eine Bewegung geben, die sie ineinander "uberf"uhrt.  
Der im Anschlu"s formulierte Satz und insbesondere sein
Beweis zeigen, wie man von diesen Annahmen ausgehend
im Rahmen einer Grundvorlesung "uber lineare Algebra zur Definition
von L"angeneinheiten und  Skalarprodukten und 
von dort nat"urlich  auf dem
"ublichen Wege  auch zu einem
Beweis des Satzes von Pythagoras gelangen kann.
Ich hebe das deshalb besonders hervor, da bei dem
hier vorgeschlagenen
Zugang der Satz des Pythagoras anders als "ublich
 beim "Ubergang vom Anschauungsraum zu seinem
mathematischen Modell in keiner Weise eingeht.
Als ersten Zugang zum Skalarprodukt w"urde ich 
den hier vorgeschlagenen Weg  dennoch nicht empfehlen.
Ich stelle mir vielmehr vor, da"s er zu einem sp"ateren Zeitpunkt 
als Anwendung und Illustration des Konzepts von Gruppenwirkungen
seinen Platz haben k"onnte, oder auch als Vortragsthema 
in einem Proseminar.
Man mag diese Arbeit als eine Randnotiz zu Kleins Erlanger Programm \cite{KlE}
lesen. Einen sehr viel radikaleren Zugang im ebenen 
Fall entwickelt Bachmann in \cite{BaSp}: Er 
konstruiert die euklidische Ebene aus ihrer
Bewegungsgruppe zusammen mit der
Teilmenge der Spiegelungen, 
also  aus einem Datum bestehend aus
einer Gruppe mit einer ausgezeichneten Teilmenge.
Der hier vorgeschlagene Zugang ist bescheidener. Nach diesen
Vorbemerkungen
fangen wir nun nocheinmal von vorne an, nur diesmal in Formeln und 
ausgehend von 
einem  reellen 
Vektorraum, etwa dem  Richtungsraum unseres affinen Anschauungsraums.

  \begin{Definition}\label{DeDr}
    Unter einem {\bf Strahl} $L$ in einem reellen Vektorraum $V$ verstehen
    wir eine Teilmenge $L\subset V$ mit der Eigenschaft, da"s es einen
    Vektor $v\neq 0$ gibt mit $L=\DR_{\geq 0}v.$   
Unter einer  {\bf Halbebene} $H$ in einem reellen Vektorraum $V$ verstehen
    wir eine Teilmenge $H\subset V$ mit der Eigenschaft, da"s es 
linear unabh"angige 
    Vektoren $v,w\in V$ gibt mit $H=\DR v+\DR_{\geq 0} w.$  
Unter dem {\bf Rand einer Halbebene} verstehen wir die einzige darin 
enthaltene Gerade; f"ur die Halbebene $H=\DR v+\DR_{\geq 0} w$ w"are das also
die Gerade $\DR v.$
Unter einer {\bf Drehgruppe}\index{Drehgruppe} in einem dreidimensionalen
reellen Vektorraum $V$ verstehen wir eine Untergruppe seiner 
Automorphismengruppe 
$$D\subset \op{GL}(V)$$
mit der Eigenschaft, da"s es f"ur je zwei Paare
 von Teilmengen unseres  Vektorraums bestehend aus einer
Halbebene und einem Strahl auf ihrem Rand genau ein Element unserer 
Untergruppe gibt, die  das eine Paar in das andere "uberf"uhrt.
Die Elemente einer solchen Drehgruppe bezeichnen wir dann auch als 
{\bf Drehungen}.\index{Drehung}
\end{Definition}
%\begin{Definition}\label{SKPn}
%Ein {\bf Skalarprodukt}\index{Skalarprodukt} 
%auf einem reellen Vektorraum $V$ 
%ist eine bilineare Abbildung
%$b:V\times V\ra\DR$  derart,
%da"s gilt $b( v,w)=b(w,v)$ f"ur alle $v,w\in V$
%und $b( v,v)\leq 0\RA v=0.$  
%Analog definiert man Skalarprodukte 
%auf einem Vektorraum "uber einem beliebigen angeordneten K"orper.
%\end{Definition} 

\begin{Satz}[{\bf Drehgruppen und Skalarprodukte}] 
Gegeben ein dreidimensionaler reeller 
Vektorraum $V,$ liefert die
Abbildung $b\mapsto\op{SO}(V;b)$ eine Bijektion
\begin{displaymath}
\left\{ 
\text{Skalarprodukte auf $V$}
 \right\}/\DR_{>0} 
\;\;\overset{\sim}{\rightarrow} \;\;
\left\{
\text{Drehgruppen $D\subset\op{GL}(V)$}
 \right\}
\end{displaymath} 
\end{Satz}

\begin{Bemerkung}
Nehmen wir unsere Drehgruppe $D$ als abgeschlossen an, so kann man
ihre Kompaktheit aus der Existenz eines kompakten homogenen Raums
folgern und ein invariantes Skalarprodukt durch Integration erhalten.
Das Ziel der folgenden Argumente ist es, einen elementareren Weg 
aufzuzeigen, der bereits im Rahmen der Grundvorlesungen gangbar ist und
zu einem besseren Verst"andnis der
beteiligten Konzepte f"uhren mag, indem er eine 
in beiden Richtungen gangbare Br"ucke
zwischen dem algebraisch besonders einfachen Konzepts eines
Skalarprodukts und dem der Anschauung besonders gut zug"anglichen
 Begriff einer Drehgruppe bereitstellt. 
Aus dem vorhergehenden Satz
folgt insbesondere f"ur jeden dreidimensionalen reellen 
Vektorraum $V$ mit einer ausgezeichneten Drehgruppe die G"ultigkeit 
des pythagoreischen Lehrsatzes in der folgenden Gestalt:
Stehen zwei Vektoren $v,w\in V$ aufeinander senkrecht in dem Sinne, da"s 
es eine Drehung gibt, die den einen festh"alt und den anderen
auf sein Negatives abbildet, und werden die Vektoren $v,w$ und $v+w$
 durch Drehungen auf die Vielfachen $an,$ $bn$ und $cn$ eines 
festen Vektors $n\neq 0$ abgebildet, so gilt
$$a^2+b^2=c^2$$
In der Tat liefert  der vorhergehende Satz 
ein unter unserer
Drehgruppe invariantes Skalarprodukt $b$ mit $b(n,n)=1.$
Aus der eben pr"azisierten Orthogonalit"atsbedingung folgt
$b(v,w)=b(-v,w),$ also $b(v,w)=0.$ Damit erhalten wir dann wie
"ublich $c^2=b(v+w,v+w)=b(v,v)+b(w,w)=a^2+b^2.$
In allen B"uchern zur linearen Algebra, die ich studiert habe,
wird  der Satz des Pythagoras 
in seiner aus der Schule bekannten Gestalt
vorausgesetzt, 
um von dort ausgehend die Br"ucke von der Anschauung zur abstrakten Theorie
euklidischer Vektorr"aume zu schlagen. 
Ich denke jedoch, da"s der Satz des Pythagoras auch eine Diskussion 
und Pr"azisierung 
im Rahmen des Studiums verdient. 
Mir selbst gef"allt die hier
vorgeschlagene Pr"azisierung 
recht gut. Der Beweis gef"allt mir weniger, aber 
ich habe keinen besseren finden k"onnen.
\end{Bemerkung}

\begin{proof}
Den 
Nachweis, da"s f"ur jedes Skalarprodukt $b$ auf $V$ die Gruppe
$\op{SO}(V;b)=\{d\in\op{GL}(V)\mid b(dv,dw)=b(v,w)\;\forall v,w\in V\text{ und }
\op{det}d=1\}$ in der Tat eine Drehgruppe  
ist,
 "uberlasse ich dem Leser  und
beginne gleich mit der Konstruktion der Umkehrabbildung.
Sei also $V$ ein dreidimensionaler reeller Vektorraum und $D\subset \op{GL}(V)$
eine Drehgruppe im Sinne unserer Definition.
Gegeben $v,w\in V,$  vereinbaren wir  die Sprechweise,
$w$ {\bf stehe senkrecht auf} $v$ 
oder auch {\bf sei orthogonal zu} $v$ 
und schreiben $$w\perp v$$
genau dann, wenn es eine Drehung $r\in D$ gibt mit $r(w)=-w$ und $r(v)=v.$ 
  Aus $v \perp w$ folgt leicht $dv\perp dw$
  f"ur jede Drehung $d$ und $\lambda v \perp \mu w$ f"ur alle
  $\mu,\lambda\in\DR.$ Des weiteren steht nur 
der Nullvektor auf sich selber senkrecht.
Um  $(w\perp v)\RA (v\perp w)$ zu zeigen, m"ussen wir  etwas weiter
ausholen. 
Gegeben linear unabh"angige Vektoren  $v,w\in V,$ 
vereinbaren wir 
 f"ur das folgende
die Notation 
$$[v,w]=( \mathbb{R} v
    + \mathbb{R}_{\geq 0} w,\mathbb{R}_{\geq 0} v)$$
f"ur das dadurch bestimmte Paar
aus einer Halbebene nebst einem Strahl auf ihrem Rand.
Unsere Definition einer Drehgruppe besagt in dieser Notation,
da"s es f"ur je zwei Paare $(v,w)$ und $(v',w')$ von 
linear unabh"angigen Vektoren genau ein Element $r$
unserer Drehgruppe gibt mit $r:[v,w]\mapsto [v',w'].$
  \begin{Lemma}\label{Ubh}
    Seien $v,w\in V$ linear unabh"angig.
    \begin{enumerate}
\item
F"ur die Drehung $r$ mit $r:[v,w]\mapsto [-v,w]$ haben wir
      $r^2=\op{id}$ und  $rv=-v.$
\item
F"ur die Drehung $r$ mit $r:[v,w]\mapsto [v,-w]$ haben wir
      $r^2=\op{id}$ und  $rv=v,$ und die einzige weitere Drehung 
$s$ mit $sv=v$ und $s(\DR v+\DR w)=\DR v+\DR w$ ist die Identit"at.
    \item 
F"ur die Drehung $r$ mit $r:[v,w]\mapsto [w,v]$ haben wir
      $r^2=\op{id}$ und es gibt $\lambda>0$ mit $rv=\lambda w$ und $r\lambda
      w=v.$
    \end{enumerate}
\end{Lemma}
\begin{proof}
1. Aus  $r^2:[v,w]\mapsto [v,w]$ folgt $r^2=\op{id},$
so da"s f"ur $r$ nur die Eigenwerte $\pm 1$ in Frage kommen.
Aus $rv\in \DR_{<0}v$ folgt dann $rv=-v.$ 
\\[2mm]
\noindent
2. Aus  $r^2:[v,w]\mapsto [v,w]$ folgt $r^2=\op{id},$
so da"s f"ur $r$ nur die Eigenwerte $\pm 1$ in Frage kommen.
Aus $rv\in \DR_{>0}v$ folgt dann $rv=v.$ F"ur eine weitere Drehung $s$ 
wie oben kommen nur die beiden M"oglichkeiten 
$sw\in \DR v+\DR_{>0}w$ und $sw\in \DR v+\DR_{<0}w$ in Betracht.
Im ersten Fall folgt $s:[v,w]\mapsto [v,w],$ also $s=\op{id},$
im letzteren Fall  $s:[v,w]\mapsto [v,-w],$ also $s=r.$
\\[2mm]
\noindent
3.
  Die Restriktion von $r$ auf die Ebene $\DR v+\DR w$  hat negative
  Determinante, 
da ihre Matrix in der Basis $v,w$ oben links eine Null hat und
  in der Nebendiagonalen positive Eintr"age.  Damit hat unsere Matrix zwei
  verschiedene reelle Eigenwerte und $r^2$ hat zwei positive reelle Eigenwerte,
  etwa mit Eigenvektoren $n$ und $m.$ Wegen $r^2:[n,m]\mapsto[n,m]$ folgt
dann  $r^2=\op{id}$. Der Rest des Lemmas ergibt sich leicht.
\end{proof}

  \begin{Lemma}\label{SeSy}    
Es gilt $w \perp v\;\RA\;v\perp w.$
  \end{Lemma}
     \begin{proof}
      F"ur  linear
      unabh"angige $v,w$ folgt das unmittelbar aus dem 
dritten Teil des vorhergehenden Lemmas \ref{Ubh}.
Gilt $w \perp v$ f"ur linear abh"angige Vektoren, so mu"s mindestens einer 
der Nullvektor sein. Im Fall $v=0$ ist $v\perp w$ offensichtlich, 
bereits die Identit"at h"alt dann 
$w$ fest und bildet  $v$ auf sein Negatives ab.
Es reicht also, wenn wir  $ v \perp 0$ zeigen f"ur
     alle $v\neq 0,$ und das folgt unmittelbar aus dem
ersten Teil des vorhergehenden Lemmas \ref{Ubh}.
  \end{proof}
  \begin{Lemma}\label{WSDh}
   Gegeben ein von Null verschiedener Vektor $n\neq 0,$ bilden die
darauf senkrecht stehenden Vektoren  eine Ebene und es gibt
genau eine Drehung $r_n$ mit $r_nn=n$ und $u\perp n\;\RA\; r_n u=-u.$
  \end{Lemma}

 \begin{proof}
Gegeben  eine  Gerade $G$ in 
einer  Ebene $ E\subset V,$ gibt es nach dem zweiten Teil
von Lemma \ref{Ubh} 
genau eine Drehung,
die die Gerade $G$ punktweise festh"alt 
und die Ebene $E$ zwar in sich selbst "uberf"uhrt, aber nicht
punktweise festh"alt. Diese Drehung ist dann nat"urlich ihr eigenes Inverses.
Die Menge der auf allen Vektoren von
$G$ senkrecht stehenden Vektoren von $E$ mu"s also
eine Gerade $G'$ sein,
n"amlich der $(-1)$-Eigenraum dieser Drehung geschnitten mit  $E.$ 
Nach Lemma \ref{SeSy} ist die Menge der auf allen Vektoren von
$G'$ senkrecht stehenden Vektoren von $E$ dann wieder 
unsere urspr"ungliche 
Gerade $G.$ 
Gegeben linear unabh"angige Vektoren $v,w\in E$ mit $v\perp w,$ 
hat die Drehung $d$ mit $d:[v,w]\mapsto [w,-v]$
folglich die Eigenschaft $d(w)\in\DR v$ und dann
sogar $d(w)\in\DR_{<0} v,$ und es ergibt sich sofort 
$d: [w,-v]\mapsto [-v, -w],$ also $d^4=\op{id}$.
Wir folgern $d^2v=-v,$ $d^2w=-w$ und damit $d^2u=-u$ f"ur alle
$u\in E.$ Andererseits haben wir $d^2\neq -\op{id},$ etwa da die Determinante
eines Quadrats nie negativ sein kann, folglich hat  $d^2$ einen von Null
verschiedenen Fixvektor $n$ und wir erhalten $E\subset\{u\in V\mid n\perp u\}.$
Hier mu"s sogar Gleichheit gelten, da sonst der Schnitt der
rechten Seite mit einer geeigneten $n$ umfassenden Ebene 
echt gr"o"ser w"are als eine Gerade.
Wir erkennen so, da"s die auf einem vorgegebenen 
von Null verschiedenen Vektor $n\neq 0$ 
senkrechten Vektoren stets eine Ebene bilden, und da"s es 
dazu stets eine Drehung
$r_n$ gibt mit $r_n(n)=n$ und $u\perp n\RA r_n(u)=-u.$ 
\end{proof}

\begin{Ubung}
 Gegeben ein Vektor $n\neq 0,$ gilt f"ur jede 
Drehung $d$ die Identit"at  $r_{dn}=d
  \circ r_n \circ d^{-1},$ und  f"ur jeden von Null verschiedenen Skalar
  $\lambda\in\DR^\times$ haben wir  $r_{\lambda n}=r_n.$
\end{Ubung}

  \begin{Lemma}\label{DLuu}
    Bildet eine Drehung einen Strahl bijektiv auf sich selbst ab, so h"alt
    sie ihn bereits punktweise fest.
\end{Lemma}
\begin{Bemerkung}
Dies Lemma formalisiert die Erfahrungstatsache, da"s eine
Achse beim Drehen ihre L"ange nicht "andert, 
und es mag
l"acherlich wirken, das beweisen zu wollen.
In der Tat h"atten wir 
diese Aussage  auch
als zus"atzliche Bedingung zu unserer Definition  
des Begriffs einer Drehgruppe mit hinzunehmen k"onnen.
Da"s ich das nicht getan habe, hat rein "asthetische Gr"unde: 
Wir k"onnen so eine gr"o"sere Wegstrecke mit reiner Logik 
zur"ucklegen. 
\end{Bemerkung}
\begin{proof}
Es gilt f"ur $u\neq 0$ und jede Drehung $d\in D$ 
  zu zeigen
  $$d(\DR_{\geq 0} u)=\DR_{\geq 0} u\;\RA\; du=u.$$
Die Idee des Beweises ist rasch erkl"art: Wir schreiben unsere Drehung
als die Komposition von zwei Drehungen, die jeweils $u$ 
auf sein Negatives abbilden.
  Dazu  w"ahlen
  wir $v\neq 0$ mit $v\perp u.$ Gilt $dv\in \DR v,$ so folgt $d^2 v\in
  \DR_{>0}v$ und damit $d^2:[u,v]\mapsto [u,v]$ und so $d^2=\op{id}$ und dann
  $du=u.$ Sonst spannen $v$ und $dv$ die zu $u$ senkrechte Ebene auf.  Nach
dem letzten Teil von Lemma  
\ref{Ubh} gibt es $\lambda>0$ und eine Drehung $r,$ die $\lambda v$ mit $dv$
  vertauscht und deren Quadrat die Identit"at ist. Daraus folgt leicht  $ru=-u.$
  F"ur die Verkn"upfung $r r_v$ gilt dann $\lambda v\mapsto dv$ und $u\mapsto
  u,$ woraus folgt $r r_v: [u,v]\mapsto [u,dv],$ also $r r_v=d,$ und damit dann
  $du=u$ wie gew"unscht.
\end{proof}
%\begin{figure}[p]\centering
%\includegraphics[width=13cm]{SkriptenBilder/Bildavw}\\[4mm]
%\noindent 
%Diese Abbildung illustriert die Definition von $a_v(w).$
%Im hier dargestellten Fall h"atten wir etwa $a_v(w)=3$ und f"ur das
%Skalarprodukt $b_v$ h"atten wir $b_v(v,w)=3/2.$
%\end{figure}
%\begin{figure}[p]\centering
%\includegraphics[height=10cm]{SkriptenBilder/Bildarv}\\[4mm]
%\noindent 
%Illustration der Identit"at $\alpha_v (w)= \alpha_w(v)$ unter der Annahme,
%da"s es eine Drehung $r$ gibt, die $v$ und $w$ vertauscht.
%\end{figure}
\begin{Lemma}
Jede Bahn einer Drehgruppe trifft jeden Strahl in genau einem Punkt. 
\end{Lemma}
\begin{proof}
Da"s jede Bahn jeden Strahl in h"ochstens einem Punkt trifft, folgt sofort aus
Lemma
\ref{DLuu}. Da"s jede Bahn jeden Strahl in mindestens einem Punkt trifft,
folgt unmittelbar aus unserer Definition einer Drehgruppe.
\end{proof}
\begin{Definition}
Eine $D$-Bahn $l\subset (V\setminus 0)$ nennen wir auch eine
{\bf L"angeneinheit}.
Gegeben eine L"angeneinheit $l$ erkl"aren wir 
die zugeh"orige Norm
  $$
\begin{array}{ccl}
V&\ra&\DR_{\geq 0}\\
v&\mapsto&\|v\|_l=\|v\|
\end{array}
$$
durch die Vorschrift $\|\lambda v\|=\lambda$ f"ur alle $\lambda\geq 0$
und $v\in l.$ 
\end{Definition}
\noindent
Nach diesen Vorbereitungen 
machen wir uns nun an den eigentlichen Beweis des Satzes und
konstruieren  ein Skalarprodukt.
Gegeben linear unabh"angige $v,w,$  gilt f"ur unser $r_v$ aus Lemma
\ref{WSDh} sicher
$r_v w=\alpha v-\gamma w$ mit $\gamma\geq 0.$ Wegen $r_v^2w=\alpha v-\alpha
\gamma v+\gamma^2w=w$ folgt 
$\gamma=1.$ Gegeben ein festes $v\neq 0,$  gibt es folglich f"ur alle $w\in V$ 
genau eine reelle Zahl $\alpha_v(w)$
mit der Eigenschaft
$$r_vw+w=\alpha_v(w)v.$$ Man erkennt unschwer, da"s 
$\alpha_v$  eine Linearform auf  
$V$ ist.
Wir k"onnen $\alpha_v$ auch charakterisieren als die eindeutig bestimmte
Linearform, die auf $v$ den Wert $2$ annimmt und auf allen zu $v$ 
senkrechten Vektoren den Wert Null.
Unsere Definitionen liefern  f"ur jede weitere Drehung $d$ 
die Identit"at $\alpha_{dv}\circ d =\alpha_v$ 
alias $\alpha_{dv} =\alpha_v\circ d^{-1}$
und f"ur jeden von Null verschiedenen Skalar $\lambda\in \DR^\times$
die Identit"at $\alpha_{\lambda v} =\lambda^{-1} \alpha_v.$
Werden zwei von Null verschiedene Vektoren $v,w$ 
durch eine Drehung miteinander vertauscht,  so gilt f"ur unsere
Ausdr"ucke weiter  die Identit"at 
$\alpha_v (w)= \alpha_w(v).$ In der Tat, aus $rv=w$ und $rw=v$ folgt
$r_w r=r r_v,$ und die von der Mitte ausgehend 
zu entwickelnde Gleichungskette 
$$\alpha_w(v)w-v=r_w(v)=r_w r w=r r_v w=r(\alpha_v(w)v-w)=\alpha_v(w)w-v$$
liefert dann die Behauptung.
Nun w"ahlen wir eine L"angeneinheit $l$ und erkl"aren eine
Abbildung
$b=b_l: V\times V\ra\DR$ durch die Vorschrift
$$b(v,w)=\left\{\begin{array}{cl} \|v\|^{2}\alpha_v(w)/2 & v\neq 0;\\ 0 & v=0. 
\end{array}\right.$$
Offensichtlich gilt  $\|v\|^{2}=b(v,v),$ und
$w\mapsto b(v,w)$ ist linear f"ur alle $v.$ 
Schlie"slich beachten wir, da"s f"ur je zwei von Null verschiedene 
Vektoren $v,w\in V$ 
die Vektoren $\|v\|^{-1}v$ und $\|w\|^{-1}w$ durch eine Drehung 
miteinander vertauscht werden. 
Nach dem vorhergehenden  folgt 
$\|v\|\|w\|^{-1}\alpha_v(w)=\|w\|\|v\|^{-1}\alpha_w(v)$ alias 
$\|v\|^{2}\alpha_v(w)=\|w\|^{2}\alpha_w(v)$ und damit 
$b(v,w)=b(w,v)$ erst f"ur je zwei von Null verschiedene 
Vektoren, aber dann auch sofort f"ur alle $v,w\in V.$
Folglich ist $b$ ein Skalarprodukt
auf $V,$ und wir haben wie versprochen eine Abbildung in die Gegenrichtung
konstruiert. Da"s unsere beiden Abbildungen 
in der Tat zueinander invers sind, mag der Leser selbst pr"ufen. 
\end{proof}

\end{document}